\documentclass[preprint,12pt]{elsarticle}

\usepackage{amsmath}
\usepackage{amscd}
\usepackage{amsthm}
\usepackage{ upgreek }
\usepackage{amssymb}
\usepackage{hyperref}
\usepackage{enumitem}
\usepackage{url}
\usepackage{times}
\newtheorem{theorem}{Theorem}
\newtheorem{lemma}{Lemma}[section]
\newtheorem{remark}{Remark}
\numberwithin{equation}{section}
\newtheorem{theoremA}{Theorem}

\allowdisplaybreaks

\journal{Journal of Number Theory}

\begin{document}

\begin{frontmatter}



\title{On the distribution of the strongly multiplicative function $2^{\omega(n)}$ on the set of natural numbers}
		
		 
		\author[label1]{K. Venkatasubbareddy}
		 \author[label2]{A. Sankaranarayanan}
		 	\affiliation[label1]{organization={University of Hyderabad},
		 			addressline={CR Rao Road, Gachibowli}, 
		 			city={Hyderabad},
		 			postcode={500046}, 
		 			state={Telangana},
		 			country={India}}
		\affiliation[label2]{organization={University of Hyderabad},
			addressline={CR Rao Road, Gachibowli}, 
			city={Hyderabad},
			postcode={500046}, 
			state={Telangana},
			country={India}}
		
\begin{abstract}
In this paper, we study the distribution of the sequence of integers $2^{\omega(n)}$ under the assumption of the strong Riemann hypothesis, where $\omega(n)$ denotes the number of distinct prime divisors of $n$. We provide an asymptotic formula for the sum $\displaystyle\sum_{n\leq x}2^{\omega(n)}$ under this assumption. We study the sum $\displaystyle\sum_{n\leq x}2^{\omega(n)}$ unconditionally too.
\end{abstract}


\end{frontmatter}



\section{Introduction}
An important question in Analytic Number Theory is to estimate the error term $E(x)$ in the asymptotic formula 
\begin{equation*}
    \sum_{n\leq x}d(n)=x\log x+(2\gamma-1)x+E(x),
\end{equation*}
where $d(n)$ is the divisor function, and this is known as the Dirichlet divisor problem. In the middle of 19th century, Dirichlet showed by elementary arguments that $E(x)\ll x^{\frac{1}{2}}$ (see \cite{Ivic2003}). Later, many researchers have improved this estimate and the best-known bound is $E(x)\ll x^{\frac{131}{416}+\varepsilon}$, which is due to Huxley \cite{Huxley}. The omega results for this error term state that the best possible bound that one can establish is:

By Hardy \cite{Ivic2003}, the omega result for the error term $E(x)$ is
 \begin{equation*}
    E(x)=\begin{cases}
        \Omega_+\left((x\log x)^\frac{1}{4}\log \log x\right)\\
        \Omega_-(x^\frac{1}{4}).
         \end{cases}
\end{equation*}
 Conjecturally, we have $E(x)=O\left(x^{\frac{1}{4}+\varepsilon}\right)$. The Hardy omega result was consequently improved by Gangadharan, Corr{\'a}di and K{\'a}tai, Hafner (see \cite{Corradi and Katai, Gangadharan, Hafner}) and for the best known bound, see \cite{Hafner}.

An arithmetical function $g(n)$ is said to be \textit{strongly multiplicative} if $g(p^a)=g(p)^a$ for all primes $p$ and all natural numbers $a$. The arithmetical function $\omega(n)$ is defined as $\omega(1)=0$ and for $n>1$,  $\omega(n)$ denotes the number of distinct prime factors of $n$. We define the arithmetical function $f(n)$ as $f(n)=2^{\omega(n)}$ for all $n\geq 1$. Note that $\omega(n)$ is additive, that is, $\omega(mn)=\omega(m)+\omega(n)$, whenever $(m,n)=1$ and we see that $f(n)$ is not only multiplicative but also strongly multiplicative. It is not difficult to see that the function $\frac{\phi(n)}{n}$ is also strongly multiplicative. 

The aim of this article is to concentrate on the distribution of the function $2^{\omega(n)}$ over natural numbers. One of the main reasons to study $2^{\omega(n)}$ here is that the function $2^{\omega(n)}$ agrees with the divisor function $d(n)$ whenever $n$ is squarefree (we call an integer $n$ squarefree if $p^2\nmid n$ whenever $p|n$). 

We define the $L$-function $F(s)$ attached to the coefficients $f(n)$, namely
\begin{equation}
    F(s)=\sum_{n=1}^\infty \frac{f(n)}{n^s}.\label{E1.1}
\end{equation}
Since $f(n)$ is multiplicative, we can write the Euler product representation of $F(s)$ as 
\begin{equation}
    F(s)=\prod_p(1+2p^{-s}+2p^{-2s}+2p^{-3s}+2p^{-4s}+\cdots)=\frac{\zeta^2(s)}{\zeta(2s)} \label{E1.2}
\end{equation}
for $\Re (s)>1$ and $F(s)$ converges absolutely and uniformly for $\Re (s)>1$.

We have the Riemann zeta function defined by
\begin{equation*}
    \zeta(s)=\sum_{n\geq1}\frac{1}{n^s}=\prod_p\left(1-\frac{1}{p^s}\right)^{-1}
\end{equation*}
for $\Re(s)>1$, which has meromorphic continuation to the whole of complex plane $\mathbb{C}$ by the functional equation
\begin{equation}
    \zeta(s)=\chi(s)\zeta(1-s)\label{E1.3}
\end{equation}
and the conversion factor $\chi(s)$ satisfies
\begin{equation}
    \mid\chi(s)\mid\sim |t|^{\frac{1}{2}-\sigma}\label{E1.4}
\end{equation} 
for $s=\sigma+it$,  $|t|\geq t_0>10$ (see \cite{Titchmarsh}).

The conjecture {\em Riemann hypothesis} states that all nontrivial zeros of the Riemann zeta function $\zeta(s)$ lie on the critical line $\Re(s) = \frac{1}{2}$. The {\em stronger version of the Riemann hypothesis} states that all the nontrivial zeros of $\zeta(s)$ lie on the critical line and each such zero is simple.


In the paper \cite{KVAS}, we considered the same arithmetical function $f(n)$ as defined in \eqref{E1.1}, but over the squareful integers and have established the following theorems.
\begin{theoremA}
    For any $\varepsilon>0$, unconditionally, we have
\begin{align*}
    \sum_{\substack{n\leq x\\\text{$n$ is squareful}}} 2^{\omega (n)}=&\mathcal{C}_1 x^\frac{1}{2}\log x+\mathcal{C}_2 x^\frac{1}{2}+\mathcal{C}_3 x^\frac{1}{3}\log x+\mathcal{C}_4 x^\frac{1}{3}\\
    &\qquad+O\left(x^\frac{1}{4}\exp\left\{c\varepsilon\left(\frac{\log x}{\log \log x}\right)^\frac{1}{3}\right\}\right)
\end{align*}
holds for some $c>0$.\label{TA}
\end{theoremA}
\begin{theoremA}
Assuming the strong Riemann hypothesis, namely, all the nontrivial zeros of $\zeta(ls)$ are on the respective critical lines for $l=2,\ 3,\ \cdots,\ 14$ and each such zero is simple, we have 
\begin{align*}
   \sum_{\substack{n\leq x\\\text{$n$ is squareful}}} 2^{\omega (n)}
    =&\mathcal{D}_1x^\frac{1}{2}\log x+\mathcal{D}_2 x^\frac{1}{2}+\mathcal{D}_3x^\frac{1}{3}\log x+\mathcal{D}_4x^\frac{1}{3}\\
    &+\sum_{\substack{\zeta(4\rho)=0,\ \Re (\rho)=\frac{1}{8},\\
0<\left| \Im (4\rho)=\gamma_\frac{1}{8}\right|<x^\frac{3}{11}}}\mathcal{D}_{1,\ \gamma_\frac{1}{8}}x^{\frac{1}{8}+i\gamma_\frac{1}{8}}\nonumber\\
&+\sum_{\substack{\zeta(5\rho)=0,\ \Re (\rho)=\frac{1}{10},\\
0<\left| \Im (5\rho)=\gamma_\frac{1}{10}\right| <x^\frac{3}{11}}}\left(\mathcal{D}_{2,\ \gamma_\frac{1}{10}}x^{\frac{1}{10}+i\gamma_\frac{1}{10}}\log x+\mathcal{D}_{3,\ \gamma_\frac{1}{10}}x^{\frac{1}{10}+i\gamma_\frac{1}{10}}\right)\nonumber\\
&+\sum_{\substack{\zeta(6\rho)=0,\ \Re (\rho)=\frac{1}{12},\\
0<\left| \Im (6\rho)=\gamma_\frac{1}{12}\right| <x^\frac{3}{11}}}\Big(\mathcal{D}_{4,\ \gamma_\frac{1}{12}}x^{\frac{1}{12}+i\gamma_\frac{1}{12}}(\log x)^2\nonumber\\
&\qquad+\mathcal{D}_{5,\ \gamma_\frac{1}{12}}x^{\frac{1}{12}+i\gamma_\frac{1}{12}}\log x+\mathcal{D}_{6,\ \gamma_\frac{1}{12}}x^{\frac{1}{12}+i\gamma_\frac{1}{12}}\Big)+O\left(x^{\frac{5}{22}+\varepsilon}\right),
\end{align*}
where $\mathcal{D}_i$'s are certain real constants and $\mathcal{D}_{i,\gamma}$'s are certain complex constants that can be evaluated explicitly. \label{TB}
\end{theoremA}

In \cite{Ivic2003}, the sum $\displaystyle\sum_{n\leq x}2^{\omega(n)}$ has been studied and it is established that, unconditionally,
\begin{equation*}
      \sum_{n\leq x}2^{\omega(n)}=\mathcal{A}_1x\log x+\mathcal{A}_2x+O\left(x^{\frac{1}{2}}\exp\left\{-C(\log x)^{\frac{3}{5}}(\log \log x)^{-\frac{1}{5}}\right\}\right),
  \end{equation*}
for some real effective constants $\mathcal{A}_1$ and $\mathcal{A}_2$, and $C>0$.

The aim of this paper is to study and establish an asymptotic formula for the above sum under the assumption of the strong Riemann hypothesis and unconditionally too.

Precisely, we prove: 

\begin{theorem}
  Let $x>0$ be large and $\varepsilon>0$. We assume the strong Riemann hypothesis, namely, all the nontrivial zeros of $\zeta(s)$ and $\zeta(2s)$ lie on the respective critical lines, and each such zero is simple. Then, we have 
  \begin{equation*}
      \sum_{n\leq x}2^{\omega(n)}=\mathcal{A}_1x\log x+\mathcal{A}_2x+\sum_{\substack{\zeta(2\rho)=0,\ \Re (\rho)=\frac{1}{4},\\
0<\left| \Im (2\rho)=\gamma_\frac{1}{4}\right|<\frac{x}{10}}}\mathcal{A}_{\gamma_\frac{1}{4}}x^{\frac{1}{4}+i\gamma_\frac{1}{4}}+O\left(x^{\varepsilon}\right),
  \end{equation*}
   where $\mathcal{A}_1=\frac{1}{\zeta(2)}$ and $\mathcal{A}_2=\frac{2\gamma-1}{\zeta(2)}$ and $\mathcal{A}_{\gamma_\frac{1}{4}}$ are some effective complex constants. \label{T1}
\end{theorem}

\begin{theorem}
Let $x>0$ be large and $\varepsilon>0$. Unconditionally, we have
   \begin{equation*}
      \sum_{n\leq x}2^{\omega(n)}=\mathcal{A}_1x\log x+\mathcal{A}_2x+\sum_{\substack{\rho=\beta+i\gamma\\
      0\leq\beta\leq \frac{1}{2}\\0<|\gamma|\leq x^{\frac{21}{29}}}}\mathop{\mathrm{Res}}_{s=\rho} F(s)\frac{x^s}{s}+O\left(x^{\frac{8}{29}+\varepsilon}\right),
  \end{equation*}
  where $\mathcal{A}_1=\frac{1}{\zeta(2)}$ and $\mathcal{A}_2=\frac{2\gamma-1}{\zeta(2)}$.\label{T2}
\end{theorem}

\begin{remark}\label{R1}
The crucial idea in proving Theorem \ref{T1} is that under the assumption of the strong Riemann hypothesis, the zeros of $\zeta(s)$ and $\zeta(2s)$ lie on the lines $\Re(s)=\frac{1}{2}$ and $\Re(s)=\frac{1}{4}$, respectively. Therefore, there are no cancellations that take place among the zeros of $\zeta(s)$ and $\zeta(2s)$ in $F(s)$, and hence each zero of $\zeta(2s)$ is a simple pole for $F(s)$. All such poles contribute together to the sum appearing in Theorem \ref{T1}. However, unconditionally, i.e., when there is no assumption on the location and nature of the zeros of $\zeta(s)$, there can be common zeros of $\zeta(s)$ and $\zeta(2s)$ in the strip $0\leq \Re(s)\leq\frac{1}{2}$, which consequently may get canceled with each other pertaining to $\frac{\zeta^2(s)}{\zeta(2s)}$. If there exist any leftover poles of $\frac{\zeta^2(s)}{\zeta(2s)}$ after possible cancellations, all such poles will contribute together to the sum appearing in Theorem \ref{T2}, for the precise nature of this sum, see Remark \ref{R3}.
\end{remark}

\begin{remark}\label{R2}
    We note that the main term $\displaystyle S_1:=\sum_{\substack{\zeta(2\rho)=0,\ \Re (\rho)=\frac{1}{4},\\
0<\left| \Im (2\rho)=\gamma_\frac{1}{4}\right|<\frac{x}{10}}}\mathcal{A}_{\gamma_\frac{1}{4}}x^{\frac{1}{4}+i\gamma_\frac{1}{4}}$ in Theorem \ref{T1} satisfies 
\begin{equation*}
    |S_1|\leq x^\frac{1}{4} \Bigg|\sum_{\substack{\zeta(2\rho)=0,\ \Re (\rho)=\frac{1}{4},\\
0<\left| \Im (2\rho)=\gamma_\frac{1}{4}\right|<\frac{x}{10}}}\mathcal{A}_{\gamma_\frac{1}{4}}x^{i\gamma_\frac{1}{4}}\Bigg|.
\end{equation*}
It is also to be noted that 
\begin{equation*}
    |S_1|\leq  \Bigg(\max_{\substack{\zeta(2\rho)=0,\ \Re (\rho)=\frac{1}{4},\\
0<\left| \Im (2\rho)=\gamma_\frac{1}{4}\right|<\frac{x}{10}}}|\mathcal{A}_{\gamma_\frac{1}{4}}|\Bigg) x^\frac{5}{4}\log x
\end{equation*} 
as the number of zeros of $\zeta(s)$ in the region $\{s=\sigma+it:0<\sigma<1,0<t\leq T\}$ is $\ll T\log T$ (see \cite{Ivic2003}). Note that $x\ll x^\frac{5}{4}$. 

Because of the oscillatory factor $x^{i\gamma_\frac{1}{4}}$ in the sum $S_1$, we expect enough cancellations to happen in the sum $S_1$. Thus, we are tempted to make the plausible conjecture that under the assumption of the strong Riemann hypothesis, we may have $S_1=O\left(x^{\frac{1}{4}+\varepsilon}\right)$, but we don't know how to prove this assertion.
\end{remark}
\begin{remark}\label{R3}
    We observe the following about the main term 
    \begin{equation*}
        S_2:=\sum_{\substack{\rho=\beta+i\gamma\\
      0\leq\beta\leq \frac{1}{2}\\0<|\gamma|\leq x^{\frac{21}{29}}}}\mathop{\mathrm{Res}}_{s=\rho} F(s)\frac{x^s}{s}
    \end{equation*} 
    appearing in Theorem \ref{T2}.
    
Here, the sum runs over the zeros $\rho=\beta+i\gamma$ of $\zeta(2s)$, the denominator of $\displaystyle F(s)=\frac{\zeta^2(s)}{\zeta(2s)}$ with $0\leq \beta\leq 1, 0<|\gamma|\leq T$. All such zeros of $\zeta(2s)$ can be zeros of $\zeta^2(s)$, which then cancel each other out in the numerator and denominator, and thus $F(s)$ may be analytic at $s=\rho$. Hence, we make the following two cases.

\textbf{Case 1}: If all the zeros $\rho=\beta+i\gamma$ of $\zeta(2s)$ with $0\leq \beta\leq 1, 0<|\gamma|\leq T$ are also zeros of $\zeta^2(s)$, then there does not exist any pole of $\displaystyle F(s)\frac{x^s}{s}$ in the rectangle $\mathcal{R}$ and the sum $S_2$ turns out to be empty, and thus $S_2=0$.

\textbf{Case 2}: If some (can be all) zeros $\rho=\beta+i\gamma$ of $\zeta(2s)$ with $0\leq \beta\leq \frac{1}{2}, 0<|\gamma|\leq T$ are left over after the cancellations, which consequently becomes the poles of certain orders of $\displaystyle\frac{\zeta^2(s)}{\zeta(2s)}$, hence the sum $S_2$ runs over all such zeros, and these observations imply the following consequences. 
\begin{enumerate}
    \item Suppose all the existing poles (say $\rho=\beta+i\gamma$) after the possible cancellations are simple. Then we have
    \begin{equation*}
        S_2=\sum_{\substack{\zeta(2\rho)=0,\zeta(\rho)\neq0\\0\leq \Re (\rho)=\beta\leq \frac{1}{2}\\0<\left| \Im (\rho)=\gamma\right|<x^{\frac{21}{29}}}}\mathcal{A}_{\beta,\gamma}x^{\beta+i\gamma},
    \end{equation*} 
    where $\displaystyle\mathcal{A}_{\beta,\gamma}x^{\beta+i\gamma}:=\mathop{\mathrm{Res}}_{s=\beta+i\gamma} \frac{\zeta^2(s)}{\zeta(2s)}\frac{x^s}{s}=\frac{x^{\beta+i\gamma}}{\beta+i\gamma}\mathop{\mathrm{Res}}_{s=\beta+i\gamma}\frac{\zeta^2(s)}{\zeta(2s)}$.
    \item Suppose that these poles are of different orders (since there can be poles of order at most $\ll \log x$ in principle in $\mathcal{R}$). Then we have
    \begin{equation*}
    S_2=\sum_{\substack{\zeta(2\rho)=0,\zeta(\rho)\neq0\\0\leq \Re (\rho)=\beta\leq \frac{1}{2}\\0<\left| \Im (\rho)=\gamma\right|<x^{\frac{21}{29}}}}\sum_{1\leq \text{ord}(\beta+i\gamma)\ll \log x}x^{\beta+i\gamma}P_{\text{ord}(\beta+i\gamma)-1}(\log x),
    \end{equation*}
    where $P_n(y)$ is a polynomial in $y$ of degree $n$, $\text{ord}(\beta+i\gamma)$ denotes the order of the pole $\rho=\beta+i\gamma$ of $\displaystyle\frac{\zeta^2(s)}{\zeta(2s)}\frac{x^s}{s}$ after the possible cancellations, and $\displaystyle x^{\beta+i\gamma}P_{\text{ord}(\beta+i\gamma)-1}(\log x) = \mathop{\mathrm{Res}}_{s=\beta+i\gamma} \frac{\zeta^2(s)}{\zeta(2s)}\frac{x^s}{s}$.
\end{enumerate}
\end{remark}

\begin{remark}\label{R4}
    By standard arguments, we find that the error term pertaining to $\displaystyle\sum_{n\leq x}2^{\omega(n)}$ is $\Omega\left(x^{\frac{1}{4}-\varepsilon}\right)$ under strong Riemann hypothesis. The interesting point of Theorem \ref{T2} is that we can take out certain main terms along with an error term $\ll x^{\frac{8}{29}+\varepsilon}$. It is important to note that the exponent $\frac{8}{29}$ is less than the exponent $\frac{131}{416}$ of the error term appearing in the Dirichlet divisor problem.
\end{remark}

\begin{remark}\label{R5}
    In principle, the ideas and arguments of this paper can be taken forward to investigate a more precise formula under the strong Riemann hypothesis for the sum $\displaystyle\sum_{n\leq x}h(n)$, where $h(n)$ is an arithmetic function such that the $L$-function associated to $h(n)$, namely, $\displaystyle H(s)=\sum_{n=1}^\infty\frac{h(n)}{n^s}$ is a quotient of Riemann zeta-functions of different arguments times a harmless Dirichlet series. For instance, under the strong Riemann hypothesis, we can obtain more precise asymptotic formulae for the sums: $\displaystyle \sum_{n\leq x}\frac{\phi(n)}{n},\ \sum_{n\leq x}d(n^2),\ \sum_{n\leq x}d(n)^2,\ \sum_{n\leq x}a(n)$, where $\displaystyle a(n)=\begin{cases}
        1&\text{ if $n$ is squareful}\\
        0&\text{otherwise}
    \end{cases}$. One needs to observe that the $L$-functions associated in these cases respectively, $\displaystyle\frac{\zeta(s)}{\zeta(s+1)},\ \frac{\zeta^3(s)}{\zeta(2s)},\ \frac{\zeta^4(s)}{\zeta(2s)}$ and $\displaystyle\frac{\zeta(2s)\zeta(3s)}{\zeta(6s)}$, which are defined in certain appropriate half-planes. For related works, we refer to \cite{Jia and AS}, \cite{KRAS1}.
\end{remark}

The following lemmas are essential to prove Theorems \ref{T1} and \ref{T2}, so we first establish them.

\section{Lemmas}
\begin{lemma}
    For $\Re(s)>1$, we have
\begin{equation*}
    F(s)=\sum_{n=1}^\infty\frac{2^{\omega(n)}}{n^s}=\frac{\zeta^2(s)}{\zeta(2s)}.
\end{equation*}\label{L1}
\end{lemma}
\begin{proof}
  For each prime $p$ and for $\Re(s)>1$, we observe that $|p^{-s}|<1$. 
  Thus, from the Euler product representation of $F(s)$, we get 
  \begin{align*}
    F(s)=&\prod_p(1+2p^{-s}+2p^{-2s}+2p^{-3s}+2p^{-4s}+\cdots)\\
    =&\prod_p\left(1-p^{-s}\right)^{-2}\left(1-p^{-2s}\right)\\
    =&\frac{\zeta^2(s)}{\zeta(2s)},
\end{align*}
which follows from the Euler product representation of the Riemann zeta function.
\end{proof}

\begin{lemma}
    Assuming the strong Riemann hypothesis, that is, that the nontrivial zeros of $\zeta(s)$ and $\zeta(2s)$ lie on the respective critical lines, we have the following residues of $F(s)\frac{x^s}{s}$ at the respective poles.
    \begin{align*}
   \mathop{\mathrm{Res}}_{s=1} F(s)\frac{x^s}{s}=&\mathcal{A}_1x\log x+\mathcal{A}_2 x,   \\
   \mathop{\mathrm{Res}}_{s=0}\ F(s)\frac{x^s}{s}=&\zeta(0)\\
    \mathop{\mathrm{Res}}_{s=\frac{1}{4}+i\gamma_\frac{1}{4}}\ F(s)\frac{x^s}{s}=&\mathcal{A}_{\gamma_\frac{1}{4}}x^{\frac{1}{4}+i\gamma_\frac{1}{4}}    
   \end{align*}
   where $\frac{1}{4}+i\gamma_\frac{1}{4}$ are the coordinates of the simple zeros of $\zeta(2s)$ on the line $\Re(s)=\frac{1}{4}$ and the $\mathcal{A}_i$'s are real effective, and the $\mathcal{A}_{\gamma_i}$'s are complex effective constants. Note that the residues at the points $s=1$ and $s=0$ hold even unconditionally. \label{L2}
\end{lemma}

\begin{proof}
    It is well-known that $\zeta(2)=\frac{\pi^2}{6}\neq0$ and the Laurent series expansion of $\zeta(s)$ about $s=1$ is given by
   \begin{equation*}
        \zeta(s)=\frac{1}{s-1}+\sum_{m=0}^\infty\frac{(-1)^m}{m!}\gamma_m(s-1)^m,
    \end{equation*}
    where $\gamma_0,\gamma_1, \gamma_2,\cdots$ are defined as
    \begin{equation*}
       \gamma_m=\lim_{n\rightarrow \infty} \left\{\left(\sum_{k=1}^n\frac{(\log k)^m}{k}\right)-\frac{(\log n)^{m+1}}{m+1}\right\}.
    \end{equation*}
     When $m=0$, $\gamma_0=\gamma$ is the usual Euler-Mascheroni constant known as the Euler constant (see \cite{Ivic2003}).

     Thus, we can write 
     \begin{align*}
          &\mathop{\mathrm{Res}}_{s=1} F(s)\frac{x^s}{s}\\
          &=\mathop{\mathrm{Res}}_{s=1}\frac{\zeta^2(s)}{\zeta(2s)}\frac{x^s}{s}\\
          &=\mathop{\mathrm{Res}}_{s=1}\frac{1}{(1+(s-1))\zeta(2s)}\left(\frac{1}{s-1}+\sum_{m=0}^\infty\frac{(-1)^m}{m!}\gamma_m(s-1)^m\right)^2(xe^{(s-1)\log x})\\
         &=\mathop{\mathrm{Res}}_{s=1}\frac{x}{\zeta(2s)}\left(\frac{1}{(s-1)^2}+\frac{2\gamma}{s-1}+(\gamma^2-2\gamma_1)+(\text{higher degree terms})\right)\times\\
         &\quad\left(1+(s-1)\log x+\frac{((s-1)\log x)^2 }{2!}+\cdots\right)\left(1-(s-1)+(s-1)^2-\cdots\right)\\
         &=\mathop{\mathrm{Res}}_{s=1}\frac{1}{\zeta(2s)}\left\{\frac{x}{(s-1)^2}+\frac{x(\log x+2\gamma-1)}{s-1}+\text{higher degree terms}\right\}\\
         &=\frac{1}{\zeta(2)}(x\log x+(2\gamma-1)x)\\
         &:=\mathcal{A}_1x\log x+\mathcal{A}_2 x,
     \end{align*}
     where $\mathcal{A}_1=\frac{1}{\zeta(2)}$ and $\mathcal{A}_2=\frac{2\gamma-1}{\zeta(2)}$.
     
     Now, for any simple zero $\frac{1}{4}+i\gamma_{\frac{1}{4}}$ of $\zeta(2s)$ on the line $\Re(s)=\frac{1}{4}$, we see that $\zeta^2(s)$ is analytic at $s=\frac{1}{4}+i\gamma_\frac{1}{4}$ and hence $F(s)\frac{x^s}{s}=\frac{\zeta^2(s)}{\zeta(2s)}\frac{x^s}{s}$ has a simple pole at $\frac{1}{4}+i\gamma_\frac{1}{4}$. Therefore,
\begin{align*}
    \mathop{\mathrm{Res}}_{s=\frac{1}{4}+i\gamma_\frac{1}{4}} F(s)\frac{x^s}{s}&=\zeta^2\left(\frac{1}{4}+i\gamma_\frac{1}{4}\right)\frac{x^{\frac{1}{4}+i\gamma_\frac{1}{4}}}{\frac{1}{4}+i\gamma_\frac{1}{4}}\mathop{\mathrm{Res}}_{s=\frac{1}{4}+i\gamma_\frac{1}{4}}\ \frac{1}{\zeta(2s)}\\
    &:=\mathcal{A}_{\gamma_\frac{1}{4}}x^{\frac{1}{4}+i\gamma_\frac{1}{4}}, 
\end{align*}
where $\mathcal{A}_{\gamma_\frac{1}{4}}$ is some complex constant which can be evaluated explicitly.

Note that $\zeta(0)=-\frac{1}{2}\neq0$ and hence $F(s)=\frac{\zeta^2(s)}{\zeta(2s)}$ is analytic at $s=0$. Therefore,
\begin{equation*}
    \mathop{\mathrm{Res}}_{s=0} F(s)\frac{x^s}{s}=\lim_{s\rightarrow0}(s-0)\frac{\zeta^2(s)}{\zeta(2s)}\frac{x^s}{s}=\zeta(0).
\end{equation*}
\end{proof}

\begin{lemma}
The Riemann hypothesis implies that 
    \begin{equation*}
    \zeta(\sigma+it)=O(( |t|+10)^\varepsilon),
\end{equation*}
for $\frac{1}{2}\leq\sigma\leq 1$ and $|t|\geq 10$; and 
 \begin{equation*}
    \frac{1}{\zeta(\sigma+it)}=O( ( |t|+10)^\varepsilon)
\end{equation*}
for $\frac{1}{2}<\sigma\leq 1$ and $|t|\geq 10$.\label{L3}
\end{lemma}
\begin{proof}
    See \cite{Titchmarsh}.
\end{proof}

\begin{lemma}
		For $\frac{1}{2}\leq \sigma\leq 2$, $T$-sufficiently large, there exist a $T^*\in[T,T+T^\frac{1}{3}]$ such that the bound 
		\begin{equation*}
			\log\zeta(\sigma\pm it)\ll (\log\log T^*)^2\ll(\log\log T)^2
		\end{equation*}
		holds uniformly for $\frac{1}{2}\leq\sigma\leq 2$ and thus we have 
		\begin{equation*}
			\mid \zeta(\sigma\pm it)\mid \ll \exp((\log\log T^*)^2)\ll_\varepsilon T^\varepsilon
		\end{equation*}
		on the horizontal line with $t=T^*$ uniformly for $\frac{1}{2}\leq \sigma\leq 2.$\\
        Which implies 
        \begin{equation}
            \frac{1}{|\zeta(\sigma\pm iT^*)|}\ll_\varepsilon T^\varepsilon  \label{E2.1}
        \end{equation}
        uniformly for $\frac{1}{2}\leq \sigma\leq 2.$        \label{L4}
\end{lemma}
\begin{proof}
    See \cite{KRAS}.
\end{proof}

\begin{lemma}
		For $\frac{1}{2}\leq \sigma\leq 2+\epsilon$ and $|t|\geq 10$, by Phrag{\'e}n-Lindel{\"o}f principle for vertical strips, we have 
      \begin{equation}
      \zeta(\sigma+it)\ll (\mid t\mid+10 )^{\max \{\frac{13}{42}(1-\sigma), 0\}+\epsilon}.\end{equation}
      \label{L5}
	\end{lemma}
\begin{proof}
    See \cite{Bourgain}.
\end{proof}
    
\section{Proof of Theorem \ref{T1}}
\begin{proof}
    
We assume the strong Riemann hypothesis for $\zeta(s)$ and $\zeta(2s)$. From Lemma \ref{L1}, for $\Re(s)>1$, we have
\begin{equation*}
    F(s)=\sum_{n=1}^\infty\frac{2^{\omega(n)}}{n^s}=\frac{\zeta^2(s)}{\zeta(2s)}.
\end{equation*} 
By applying Perron's formula to $F(s)$, we get
\begin{equation*}
    \sum_{n\leq x}2^{\omega(n)}=\frac{1}{2\pi i}\int_{1+\varepsilon-iT}^{1+\varepsilon+iT} F(s)\frac{x^s}{s}ds+O\left(\frac{x^{1+\varepsilon}}{T}\right),
\end{equation*}
where $10\leq T\leq x$ is a parameter to be chosen later. 

Now, we move the line of integration to $\Re(s)=-\frac{1}{2}$. Then, in the rectangle $\mathcal{R}$ formed by the vertices $\displaystyle 1+\varepsilon-iT,1+\varepsilon+iT,-\frac{1}{2}+iT,-\frac{1}{2}-iT, 1+\varepsilon-iT$ with straight line segments, $F(s)\frac{x^s}{s}$ possesses the following poles (under the assumption of the strong Riemann hypothesis), one at $s=1$ of order 2, a simple pole at each nontrivial zero $\rho=\frac{1}{4}+i\gamma_\frac{1}{4}$ of $\zeta(2s)$, which lie on the line $\Re(s)=\frac{1}{4}$ and a simple pole at $s=0$. Thus, by Cauchy's residue theorem for rectangles, we get
\begin{align}
     \sum_{n\leq x}2^{\omega(n)}=&\mathcal{A}_1x\log x+\mathcal{A}_2 x+\zeta(0)+\sum_{\rho}\mathop{\mathrm{Res}}_{s=\rho} F(s)\frac{x^s}{s}\nonumber\\
     &+\frac{1}{2\pi i}\left\{\int_{-\frac{1}{2}-iT}^{-\frac{1}{2}+iT}+\int_{-\frac{1}{2}-iT}^{1+\varepsilon-iT}+\int_{-\frac{1}{2}+iT}^{1+\varepsilon+iT}\right\} F(s)\frac{x^s}{s}ds+O\left(\frac{x^{1+\varepsilon}}{T}\right)\nonumber\\
     :=&\mathcal{A}_1x\log x+\mathcal{A}_2 x+\zeta(0)+\sum_\rho \mathop{\mathrm{Res}}_{s=\rho} F(s)\frac{x^s}{s}+I_1+I_2+I_3\nonumber\\
     &\qquad+O\left(\frac{x^{1+\varepsilon}}{T}\right),\label{E3.1}
\end{align}
where the right-hand side sum runs over all the poles of $F(s)\frac{x^s}{s}$ lying inside the rectangle $\mathcal{R}$. 

First, we deal with the vertical line integration $I_1$ as follows:
\begin{align*}
   | I_1|\ll&\int_{-\frac{1}{2}-iT}^{-\frac{1}{2}+iT}\left|\frac{\zeta^2(s)}{\zeta(2s)}\frac{x^s}{s}\right||ds|\\
    \ll&\int_{-\frac{1}{2}-iT}^{-\frac{1}{2}+iT}\frac{|\chi^2(s)\zeta^2(1-s)|}{|\chi(2s)\zeta(1-2s)|}\frac{|x^s|}{|s|}|ds|\\
    \ll&x^{-\frac{1}{2}+\varepsilon}+x^{-\frac{1}{2}}\int_{10}^T\frac{|\chi^2(-\frac{1}{2}+it)\zeta^2(1-(-\frac{1}{2}+it))|}{|\chi(-1+2it)\zeta(1-(-1+2it))|}t^{-1}dt\\
    \ll& x^{-\frac{1}{2}+\varepsilon}+x^{-\frac{1}{2}}\int_{10}^T t^{2\left(\frac{1}{2}-(-\frac{1}{2})\right)-\left(\frac{1}{2}-(-1)\right)-1}dt\\
    \ll&x^{-\frac{1}{2}+\varepsilon}T^{\frac{1}{2}},
\end{align*}
which follows from the functional equation \eqref{E1.3} of $\zeta(s)$, the approximation \eqref{E1.4}. We note here that this estimate holds even unconditionally.

Now, to evaluate the contribution from the horizontal line integrations $I_2$ and $I_3$, we split the integration $\displaystyle\int_{-\frac{1}{2}+iT}^{1+\varepsilon+iT}$ as follows:
\begin{align*}
    |I_2|+|I_3|\ll&\int_{-\frac{1}{2}+iT}^{1+\varepsilon+iT}\frac{|\zeta^2(s)|}{|\zeta(2s)|}\left|\frac{x^s}{s}\right||ds|\\
    \ll&\left\{\int_{-\frac{1}{2}+iT}^{0+iT}+\int_{0+iT}^{\frac{1}{4}+iT}+\int_{\frac{1}{4}+iT}^{\frac{1}{2}+iT}+\int_{\frac{1}{2}+iT}^{1+\varepsilon+iT}\right\}\frac{|\zeta^2(s)|}{|\zeta(2s)|}\left|\frac{x^s}{s}\right||ds|\\
    \ll&\int_{-\frac{1}{2}}^{0}\frac{|\chi^2(\sigma+iT)\zeta^2(1-(\sigma+iT))|}{|\chi(2\sigma+2iT)\zeta(1-(2\sigma+2iT))|}x^\sigma T^{-1}d\sigma\\
    &+\int_{0}^{\frac{1}{4}}\frac{|\chi^2(\sigma+iT)\zeta^2(1-(\sigma+iT))|}{|\chi(2\sigma+2iT)\zeta(1-(2\sigma+2iT))|}x^\sigma T^{-1}d\sigma\\
    &+\int_{\frac{1}{4}}^{\frac{1}{2}}\frac{|\chi^2(\sigma+iT)\zeta^2(1-(\sigma+iT))|}{|\zeta(2\sigma+2iT)|}x^\sigma T^{-1}d\sigma\\
    &+\int_{\frac{1}{2}}^{1+\varepsilon}\frac{|\zeta^2(\sigma+iT)|}{|\zeta(2\sigma+2iT)|}x^\sigma T^{-1}d\sigma\\
    \ll &\int_{-\frac{1}{2}}^{0}x^\sigma T^{2(\frac{1}{2}-\sigma)-(\frac{1}{2}-2\sigma)-1+3\varepsilon}d\sigma+\int_0^\frac{1}{4}x^\sigma T^{2(\frac{1}{2}-\sigma)-(\frac{1}{2}-2\sigma)-1+3\varepsilon}d\sigma\\
    &+\int_{\frac{1}{4}}^{\frac{1}{2}}x^\sigma T^{2(\frac{1}{2}-\sigma)-1+3\varepsilon}d\sigma+\int_{\frac{1}{2}}^{1+\varepsilon}x^\sigma T^{-1+3\varepsilon}d\sigma\\
    \ll&\int_{-\frac{1}{2}}^{0}x^\sigma T^{-\frac{1}{2}+3\varepsilon}d\sigma+\int_0^\frac{1}{4}x^\sigma T^{-\frac{1}{2}+3\varepsilon}d\sigma\\
    &+\int_{\frac{1}{4}}^{\frac{1}{2}}x^\sigma T^{-2\sigma+3\varepsilon}d\sigma+\int_{\frac{1}{2}}^{1+\varepsilon}x^\sigma T^{-1+3\varepsilon}d\sigma\\
    \ll& T^{-\frac{1}{2}+3\varepsilon}+x^{\frac{1}{4}}T^{-\frac{1}{2}+3\varepsilon}+x^\frac{1}{2}T^{-1+3\varepsilon}+x^{1+\varepsilon}T^{-1+3\varepsilon}\\
    \ll& x^{\frac{1}{4}}T^{-\frac{1}{2}+3\varepsilon}+x^{1+\varepsilon}T^{-1+3\varepsilon},
\end{align*}
which follows from the functional equation \eqref{E1.3} of $\zeta(s)$, the approximation \eqref{E1.4} and Lemma \ref{L3}.

Hence, in total, from Lemma \ref{L2} and \eqref{E3.1} we have 
\begin{align*}
      \sum_{n\leq x}2^{\omega(n)}=&\mathcal{A}_1x\log x+\mathcal{A}_2x+\sum_{\substack{\zeta(2\rho)=0,\ \Re (\rho)=\frac{1}{4},\\
0<\left| \Im (2\rho)=\gamma_\frac{1}{4}\right|<T}}\mathcal{A}_{\gamma_\frac{1}{4}}x^{\frac{1}{4}+i\gamma_\frac{1}{4}}+\zeta(0)\\
&+O\left(x^{-\frac{1}{2}+\varepsilon}T^{\frac{1}{2}}\right)+O\left(x^{\frac{1}{4}}T^{-\frac{1}{2}+3\varepsilon}\right)+O\left(x^{1+\varepsilon}T^{-1+3\varepsilon}\right).
  \end{align*}
Finally, making our choice as $T= \frac{x}{10}$, we obtain 
\begin{equation*}
      \sum_{n\leq x}2^{\omega(n)}=\mathcal{A}_1x\log x+\mathcal{A}_2x+\sum_{\substack{\zeta(2\rho)=0,\ \Re (\rho)=\frac{1}{4},\\
0<\left| \Im (2\rho)=\gamma_\frac{1}{4}\right|<\frac{x}{10}}}\mathcal{A}_{\gamma_\frac{1}{4}}x^{\frac{1}{4}+i\gamma_\frac{1}{4}}+O\left(x^{10\varepsilon}\right),
  \end{equation*}
  for some effective real constants $\mathcal{A}_1,\mathcal{A}_2$   and effective complex constants $\mathcal{A}_{\gamma_\frac{1}{4}}$.

This completes the proof of Theorem \ref{T1}.
\end{proof}

\section{Proof of Theorem \ref{T2}}
\begin{proof}
Following the same arguments as in the proof of Theorem \ref{T1}, we get
\begin{align}
     \sum_{n\leq x}2^{\omega(n)}=&\mathcal{A}_1x\log x+\mathcal{A}_2 x+\zeta(0)+\sum_{\rho}\mathop{\mathrm{Res}}_{s=\rho} F(s)\frac{x^s}{s}\nonumber\\
     &+\frac{1}{2\pi i}\left\{\int_{-\frac{1}{2}-iT}^{-\frac{1}{2}+iT}+\int_{-\frac{1}{2}-iT}^{1+\varepsilon-iT}+\int_{-\frac{1}{2}+iT}^{1+\varepsilon+iT}\right\} F(s)\frac{x^s}{s}ds+O\left(\frac{x^{1+\varepsilon}}{T}\right)\nonumber\\
     :=&\mathcal{A}_1x\log x+\mathcal{A}_2 x+\zeta(0)+\sum_\rho \mathop{\mathrm{Res}}_{s=\rho} F(s)\frac{x^s}{s}+I_1+I_2+I_3\nonumber\\
     &\qquad+O\left(\frac{x^{1+\varepsilon}}{T}\right),\label{E4.1}
\end{align}
We note that, in this case, the singularities of $F(s)\frac{x^s}{s}$ inside the rectangle $\mathcal{R}$ formed by the vertices $\displaystyle 1+\varepsilon-iT,1+\varepsilon+iT,-\frac{1}{2}+iT,-\frac{1}{2}-iT, 1+\varepsilon-iT$ with straight line segments are; a pole at $s=1$ of order 2, poles at every zero $\rho=\beta+i\gamma$ of $\zeta(2s)$ with  $0\leq \beta\leq \frac{1}{2}, 0<|\gamma|\leq T$ if any after the possible cancellations and a simple pole at $s=0$. Thus, the sum on the right-hand side of \eqref{E4.1} runs over all the singularities of $F(s)\frac{x^s}{s}$ inside the rectangle $\mathcal{R}$, mentioned above. Here, we make a special choice $T$ such that $2T=T^*$ of Lemma \ref{L4}, which satisfies \eqref{E2.1}. 

From the proof of Theorem \ref{T1}, the vertical line integration contribution is given by
\begin{equation*}
    J_1=\int_{-\frac{1}{2}-iT}^{-\frac{1}{2}+iT}F(s)\frac{x^s}{s}ds\ll x^{-\frac{1}{2}+\varepsilon}T^{\frac{1}{2}}.
\end{equation*}
We note that this contribution is unconditional. 
 
Now, for the contribution from the horizontal lines $J_2$ and $J_3$, we split the integral $\displaystyle\int_{-\frac{1}{2}+iT}^{1+\varepsilon+iT}$ as follows:
\begin{align*}
    &|J_2|+|J_3| \ll\int_{-\frac{1}{2}+iT}^{1+\varepsilon+iT}\frac{|\zeta^2(s)|}{|\zeta(2s)|}\left|\frac{x^s}{s}\right||ds|\\
    \ll&\left\{\int_{-\frac{1}{2}+iT}^{0+iT}+\int_{0+iT}^{\frac{1}{4}+iT}+\int_{\frac{1}{4}+iT}^{\frac{1}{2}+iT}+\int_{\frac{1}{2}+iT}^{1+\varepsilon+iT}\right\}\frac{|\zeta^2(s)|}{|\zeta(2s)|}\left|\frac{x^s}{s}\right||ds|\\
    \ll&\int_{-\frac{1}{2}}^{0}\frac{|\chi^2(\sigma+iT)\zeta^2(1-(\sigma+iT))|}{|\chi(2\sigma+2iT)\zeta(1-(2\sigma+2iT))|}x^\sigma T^{-1}d\sigma\\
    &+\int_{0}^{\frac{1}{4}}\frac{|\chi^2(\sigma+iT)\zeta^2(1-(\sigma+iT))|}{|\chi(2\sigma+2iT)\zeta(1-(2\sigma+2iT))|}x^\sigma T^{-1}d\sigma\\
    &+\int_{\frac{1}{4}}^{\frac{1}{2}}\frac{|\chi^2(\sigma+iT)\zeta^2(1-(\sigma+iT))|}{|\zeta(2\sigma+2iT)|}x^\sigma T^{-1}d\sigma\\
    &+\int_{\frac{1}{2}}^{1+\varepsilon}\frac{|\zeta^2(\sigma+iT)|}{|\zeta(2\sigma+2iT)|}x^\sigma T^{-1}d\sigma\\
    \ll& \int_{-\frac{1}{2}}^{0}x^\sigma T^{2(\frac{1}{2}-\sigma)-(\frac{1}{2}-2\sigma)-1+3\varepsilon}d\sigma+\int_0^\frac{1}{4}x^\sigma T^{2(\frac{1}{2}-\sigma)-(\frac{1}{2}-2\sigma)+2\times\frac{13}{42}(1-(1-\sigma))-1+2\varepsilon}d\sigma\\
    &+\int_{\frac{1}{4}}^{\frac{1}{2}}x^\sigma T^{2(\frac{1}{2}-\sigma)+2\times\frac{13}{42}(1-(1-\sigma))-1+2\varepsilon}d\sigma+\int_{\frac{1}{2}}^{1+\varepsilon}x^\sigma T^{2\times\frac{13}{42}(1-\sigma)-1+2\varepsilon}d\sigma\\
    \ll& \int_{-\frac{1}{2}}^0x^\sigma T^{-\frac{1}{2}+3\varepsilon}d\sigma+\int_{0}^\frac{1}{4}x^\sigma T^{-\frac{1}{2}+\frac{13}{21}\sigma+2\varepsilon}d\sigma+\int_\frac{1}{4}^\frac{1}{2}x^\sigma T^{-\frac{29}{21}\sigma+2\varepsilon}d\sigma\\
    &+\int_{\frac{1}{2}}^{1+\varepsilon}x^\sigma T^{-\frac{8}{21}-\frac{13}{21}\sigma+2\varepsilon}d\sigma\\
    \ll& T^{-\frac{1}{2}+3\varepsilon}+x^{\frac{1}{4}}T^{-\frac{29}{84}+2\varepsilon}+x^\frac{1}{2}T^{-\frac{29}{42}+2\varepsilon}+x^{1+\varepsilon}T^{-1+2\varepsilon},
\end{align*}
which follows from the functional equation \eqref{E1.3} of $\zeta(s)$, the approximation \eqref{E1.4}, Lemmas \ref{L4} and \ref{L5}.

Hence, in total, from Lemma \ref{L2} and \eqref{E4.1} we have 
\begin{align*}
      \sum_{n\leq x}2^{\omega(n)}=&\mathcal{A}_1x\log x+\mathcal{A}_2x+\sum_{\substack{\rho=\beta+i\gamma\\
      0\leq\beta\leq \frac{1}{2}\\0<|\gamma|\leq T}}\mathop{\mathrm{Res}}_{s=\rho} F(s)\frac{x^s}{s}+\zeta(0)+O\left(x^{-\frac{1}{2}+\varepsilon}T^{\frac{1}{2}}\right)\\
&+O\left(x^{\frac{1}{4}}T^{-\frac{29}{84}+2\varepsilon}\right)+O\left(x^{\frac{1}{2}}T^{-\frac{29}{42}+2\varepsilon}\right)+O\left(x^{1+\varepsilon}T^{-1+2\varepsilon}\right).
  \end{align*}
Finally, making our choice as $x^{\frac{1}{4}}T^{-\frac{29}{84}}=x^{\frac{1}{2}}T^{-\frac{29}{42}}$, i.e., $T=x^{\frac{21}{29}}$, we obtain 
\begin{align*}
      \sum_{n\leq x}2^{\omega(n)}=&\mathcal{A}_1x\log x+\mathcal{A}_2x+\sum_{\substack{\rho=\beta+i\gamma\\
      0\leq\beta\leq \frac{1}{2}\\0<|\gamma|\leq x^{\frac{21}{29}}}}\mathop{\mathrm{Res}}_{s=\rho} F(s)\frac{x^s}{s}+O\left(x^{\frac{8}{29}+\varepsilon}\right).
\end{align*}

This completes the proof of Theorem \ref{T2}.
\end{proof}

\noindent {\bf Acknowledgments.} The first author wishes to express his gratitude to the Funding Agency "Ministry of Education, Govt. of India" for the Prime Minister's Research Fellowship (PMRF), ID: 3701831 for its financial support.


\begin{thebibliography}{1}
\bibitem[1]{Ivic2003} A. Ivi{\'c}, \emph{The Riemann Zeta-Function: Theory and Applications}, Dover Publications, Inc., Mineola, New York, 2003.
\bibitem[2]{Bourgain} J. Bourgain, \emph{Decoupling, exponential sums and the Riemann zeta function}, J. Amer. Math. Soc., {\bf 30} (2017) 205-224.
\bibitem[3]{Corradi and Katai} K. Corrádi, and I. Kátai, \emph{A comment on K. S. Gangadharan's paper entitled ``Two classical lattice point problems''}, Magyar Tud. Akad. Mat. Fiz. Oszt. Közl.  {\bf 17} (1967), 89–97.
\bibitem[4]{Gangadharan} K. S. Gangadharan, \emph{Two classical lattice point problems}, Proc. Cambridge Philos. Soc.  {\bf 57} (1961), 699–721.
\bibitem[5]{Hafner} J. L. Hafner, \emph{New omega theorems for two classical lattice point problems}, Invent. Math.  {\bf 63} (1981), no. 2, 181–186.
\bibitem[6]{Huxley} M. N. Huxley, \emph{Exponential sums and lattice points. III}, Proc. London Math. Soc. (3)   {\bf 87} (2003), no. 3, 591–609.
\bibitem[7]{Jia and AS} C. Jia and A. Sankaranarayanan, The mean square of the divisor function. Acta Arith. \textbf{164} (2014), no. 2, 181–208.
\bibitem[8]{KRAS} K. Ramachandra and A. Sankaranarayanan, \emph{Notes on the Riemann zeta-function}, J. Indian Math. Soc. {\bf 57} (1991), no.1-4, 67–77.
\bibitem[9]{KRAS1} K. Ramachandra and A. Sankaranarayanan, On an asymptotic formula of Srinivasa Ramanujan. Acta Arith. \textbf{109} (2003), no. 4, 349–357.
\bibitem[10]{Titchmarsh} E. C. Titchmarsh and D. R. Heath-Brown, \emph{The Theory of the Riemann Zeta-function}, second edition, Clarendon Press, Oxford, 1986.
\bibitem[11]{KVAS} K. Venkatasubbareddy and A. Sankaranarayanan, \emph{On the distribution of the restricted sequence of integers $2^{\omega(n)}$}, Integers   \textbf{24} (2024), Paper No. A110, 15 pp.

\end{thebibliography}
\end{document}